\documentclass[12pt]{amsart}
\RequirePackage{mathtools}
\RequirePackage{amsopn}
\RequirePackage{amsfonts}
\RequirePackage{paralist}
\RequirePackage{amssymb}
\RequirePackage{amsthm}
\RequirePackage{mathrsfs}
\usepackage[alphabetic]{amsrefs}
\renewcommand\MR[1]{\relax} 
\usepackage{tikz}
\usetikzlibrary{cd,decorations.pathmorphing,decorations.markings}
\mathtoolsset{showonlyrefs}

\usepackage{enumitem}

\newtheorem{thm}{Theorem}[section] 
\numberwithin{equation}{section}

\newtheorem{cor}[thm]{Corollary}
\newtheorem{lemma}[thm]{Lemma}

\theoremstyle{definition}

\theoremstyle{remark}
\newtheorem{remark}[thm]{Remark}

\def\mathcs{C^{*}}
\newcommand{\cs}{\ensuremath{\mathcs}}
\def\calg{$\cs$-algebra}
\DeclareMathSymbol{\rtimes}{\mathbin}{AMSb}{"6F}
\newcommand{\ib}{im\-prim\-i\-tiv\-ity bi\-mod\-u\-le}

\newcommand{\sme}{\,\mathord{\mathop{\text{--}}\nolimits_{\relax}}\,}

\newcommand\set[1]{\{\,#1\,\}}

\def\restr#1{|_{{#1}}}
\makeatletter
\def\labelenumi{\textnormal{(\@alph\c@enumi)}}
\def\theenumi{\@alph \c@enumi}
\def\labelenumii{\textnormal{(\@roman\c@enumii)}}
\def\theenumii{\@roman \c@enumii}
\newcount\charno
\def\alphapart#1{\charno=96
\advance\charno by#1\char\charno}

\makeatother

\def\<{\langle}
\def\>{\rangle}
\let\ipscriptstyle=\scriptscriptstyle
\def\lipsqueeze{{\mskip -3.0mu}}
\def\ripsqueeze{{\mskip -3.0mu}}
\def\ipcomma{\nobreak\mathrel{,}\nobreak}
\newbox\ipstrutbox
\setbox\ipstrutbox=\hbox{\vrule height8.5pt
width 0pt}
\def\ipstrut{\copy\ipstrutbox}
\def\lip#1<#2,#3>{\mathopen{\relax_{\ipstrut\ipscriptstyle{
#1}}\lipsqueeze
\langle} #2\ipcomma #3 \rangle}
\def\blip#1<#2,#3>{\mathopen{\relax_{\ipstrut
\ipscriptstyle{ #1}}\lipsqueeze\bigl\langle} #2\ipcomma #3 \bigr\rangle}
\def\rip#1<#2,#3>{\langle #2\ipcomma #3
\rangle_{\ripsqueeze\ipstrut\ipscriptstyle{#1}}}
\def\brip#1<#2,#3>{\bigl\langle #2\ipcomma #3
\bigr\rangle_{\ripsqueeze\ipstrut\ipscriptstyle{#1}}}
\def\angsqueeze{\mskip -6mu}
\def\smangsqueeze{\mskip -3.7mu}
\def\trip#1<#2,#3>{\langle\smangsqueeze\langle #2\ipcomma #3
\rangle\smangsqueeze\rangle_{\ripsqueeze\ipstrut\ipscriptstyle{#1}}}
\def\btrip#1<#2,#3>{\bigl\langle\angsqueeze\bigl\langle #2\ipcomma
#3
\bigr\rangle
\angsqueeze\bigr\rangle_{\ripsqueeze\ipstrut\ipscriptstyle{#1}}}
\def\tlip#1<#2,#3>{\mathopen{\relax_{\ipstrut\ipscriptstyle{
#1}}\lipsqueeze \langle\smangsqueeze\langle} #2\ipcomma #3
\rangle\smangsqueeze\rangle}
\def\btlip#1<#2,#3>{\mathopen{\relax_{\ipstrut\ipscriptstyle{
#1}}\lipsqueeze
\bigl\langle\angsqueeze\bigl\langle} #2\ipcomma #3
\bigr\rangle\angsqueeze\bigr\rangle}

\def\ip(#1|#2){(#1\mid #2)}
\def\bip(#1|#2){\bigl(#1 \mid #2\bigr)}
\def\Bip(#1|#2){\Bigl( #1 \bigm| #2 \Bigr)}
\expandafter\ifx\csname BibSpec\endcsname\relax\else
\BibSpec{collection.article}{%
    +{}  {\PrintAuthors}                {author}
    +{,} { \textit}                     {title}
    +{.} { }                            {part}
    +{:} { \textit}                     {subtitle}
    +{,} { \PrintContributions}         {contribution}
    +{,} { \PrintConference}            {conference}
    +{}  {\PrintBook}                   {book}
    +{,} { }                            {booktitle}
    +{,} { }                            {series}
    +{,} { \voltext}                    {volume}
    +{,} { }                            {publisher}
    +{,} { }                            {organization}
    +{,} { }                            {address}
    +{,} { \PrintDateB}                 {date}
    +{,} { pp.~}                        {pages}
    +{,} { }                            {status}
    +{,} { \PrintDOI}                   {doi}
    +{,} { available at \eprint}        {eprint}
    +{}  { \parenthesize}               {language}
    +{}  { \PrintTranslation}           {translation}
    +{;} { \PrintReprint}               {reprint}
    +{.} { }                            {note}
    +{.} {}                             {transition}
}
\BibSpec{article}{%
    +{}  {\PrintAuthors}                {author}
    +{,} { \textit}                     {title}
    +{.} { }                            {part}
    +{:} { \textit}                     {subtitle}
    +{,} { \PrintContributions}         {contribution}
    +{.} { \PrintPartials}              {partial}
    +{,} { }                            {journal}
    +{}  { \textbf}                     {volume}
    +{}  { \PrintDatePV}                {date}
    +{,} { \eprintpages}                {pages}
    +{,} { }                            {status}
    +{,} { \PrintDOI}                   {doi}
    +{,} { available at \eprint}        {eprint}
    +{}  { \parenthesize}               {language}
    +{}  { \PrintTranslation}           {translation}
    +{;} { \PrintReprint}               {reprint}
    +{.} { }                            {note}
    +{.} {}                             {transition}
}
\BibSpec{book}{%
    +{}  {\PrintPrimary}                {transition}
    +{,} { \textit}                     {title}
    +{.} { }                            {part}
    +{:} { \textit}                     {subtitle}
    +{,} { \PrintEdition}               {edition}
    +{}  { \PrintEditorsB}              {editor}
    +{,} { \PrintTranslatorsC}          {translator}
    +{,} { \PrintContributions}         {contribution}
    +{,} { }                            {series}
    +{,} { \voltext}                    {volume}
    +{,} { }                            {publisher}
    +{,} { }                            {organization}
    +{,} { }                            {address}
    +{,} { pp.~}                        {pages}
    +{,} { \PrintDateB}                 {date}
    +{,} { }                            {status}
    +{}  { \parenthesize}               {language}
    +{}  { \PrintTranslation}           {translation}
    +{;} { \PrintReprint}               {reprint}
    +{.} { }                            {note}
    +{.} {}                             {transition}
}
\fi
\evensidemargin=\oddsidemargin
\newcommand\B{\mathscr{B}}
\newcommand\E{\mathscr{E}}
\newcommand\CC{\mathscr{C}}
\newcommand\JJ{\mathscr{L}}
\newcommand\KK{\mathscr{K}}
\newcommand\A{\mathscr{A}}

\newcommand\M{\mathscr{M}}

\newcommand\go{G^{(0)}}
\newcommand\ho{H^{(0)}}
\newcommand\ko{K^{(0)}}

\newcommand{\thmref}[1]{Theorem~\textup{\ref{#1}}}
\newcommand{\corref}[1]{Corollary~\textup{\ref{#1}}}
\newcommand{\secref}[1]{Section~\textup{\ref{#1}}}

\newcommand\II{\mathscr{I}}

\renewcommand{\:}{\colon}
\newcommand{\fb}{Fell bundle}
\newcommand{\ibm}{imprimitivity bimodule}

\newcommand{\units}{^{(0)}}

\newcommand{\pre}[1]{{}_{#1}}

\newcommand{\inv}{^{-1}}
\newcommand{\under}{\backslash}

\newcommand\Id{\operatorname{Id}}

\begin{document}
\begin{abstract}
  We use the
Ladder Technique to establish bijections between the ideals of related
Fell bundles.
\end{abstract}

\title{Fell bundle ladder}

\author[Kaliszewski]{S. Kaliszewski}
\address{School of Mathematical and Statistical Sciences
\\Arizona State University
\\Tempe, Arizona 85287}
\email{kaliszewski@asu.edu}

\author[Quigg]{John Quigg}
\address{School of Mathematical and Statistical Sciences
\\Arizona State University
\\Tempe, Arizona 85287}
\email{quigg@asu.edu}

\author[Williams]{Dana P. Williams}
\address{Department of Mathematics\\ Dartmouth College \\ Hanover, NH
  03755-3551 USA}
\email{dana.williams@Dartmouth.edu}

 \date{July 21, 2025}


\maketitle

\section{Introduction}
\label{sec:introduction}

In studying the structure of \cs-crossed products associated to
  groups or groupoids acting on a \cs-algebra, invariant ideals of the
  algebra play a key role.  Such ideals correspond naturally to ideals
  in the crossed product in such a way that the quotient of the
  crossed product by the ideal corresponds to a dynamical system for
  an action on the quotient of the \cs-algebra by the invariant ideal.
  Since it is profitable to think of the \cs-algebra of a Fell bundle
  over a groupoid as a powerful generalization of \cs-crossed
  products, it is natural to look for an analogous result for Fell
  bundle \cs-algebras.  In \cite{ion}, Marius Ionescu and the third
  author showed that one could identify suitably invariant ideals of
  the \cs-algebra obtained by restricting the Fell bundle to the unit
  space of the groupoid.  Then it was possible to obtain short exact
  sequences of \cs-algebras generalizing the cases where the there is
  a concrete action.  In \cite{kqwrieffel}, we showed that these
  generalized invariant ideals corresponded naturally to a special
  type of Fell subbundles that we called \emph{ideals} of the Fell
  bundle itself.  This made the generalization coming from \cite{ion}
  much easier to work with since the notion of an ideal in a Fell
  bundle is a purely bundle-theoretic notion which does not involve
  \cs-algebras or completions.  The main result in \cite{kqwrieffel}
  allowed us to establish an appropriate version of the Rieffel
  correspondence between ideals of equivalent \fb s.  At about the
  same time, in \cite{gkqwladder}, we, together with Matthew
  Gillespie, introduced what we call the ``Ladder Technique'', and
  used it to establish a lattice isomorphism between the sets of
  invariant ideals of a \calg\ and the crossed product by an action or
  coaction of a group.

In this paper we again employ the ladder technique, this time
to prove analogues for \fb s of our lattice isomorphism from
\cite{gkqwladder}.  More precisely, we have two related theorems: in
one, given an action of a groupoid $G$ on a \fb\ $\A$, we provide a
lattice isomorphism between the $G$-invariant ideals of $\A$ and the
ideals of the semidirect product $\A\propto G$,\footnote{In \cite{hkqw2} we introduced these semidirect-product Fell bundles using a different notation: $S(\A,G)$; we always had in mind that this notation was temporary until we decided upon a more appropriate one, which we are introducing here.} and in the other, we
give a lattice isomorphism between the ideals of any \fb\ $\A$ and the
$G$-invariant ideals of the action product $\A\rtimes G$.


In \secref{sec:prelim} we record our notation and review relevant
definitions and results concerning groupoid equivalences, \fb\
equivalences, ideals of \fb s, and the ladder technique.

\secref{sec:ladder} contains our main results, which use the
\emph{ladder technique} from \cite{gkqwladder}, and in fact
Theorems~\ref{fb ladder} and \ref{bonus ladder} have roughly the
same goal: bijections between sets of  ideals,
this time for \fb s.  In \secref{sec:struts} we build the vertical
struts for the ladders.

\section{Preliminaries}\label{sec:prelim}

Our groupoids will all be locally compact Hausdorff with open range
and source maps.  We refer to \cite[Section~2.3]{tool} for groupoid
equivalences.  Recall here that if $G$ and $H$ are groupoids, then a
space $Z$ is a $(G,H)$-equivalence if there are free and proper
commuting actions of $G$ and $H$ on $Z$ ($G$ on the left and $H$ on
the right) such that the moment map $\rho$ for the $G$-action factors
through a homeomorphism of $Z/H$ with $G\units$, and the moment map
$\sigma$ for the $H$-action factors through a homeomorphism of
$G\under Z$ with $H\units$.  We also require the moment maps to be
open, which is actually automatic as we require the range and source
maps on groupoids to be open.  If $Z$ is a $(G,H)$-equivalence then
there are continuous and open pairings
$\pre G[\cdot,\cdot]:Z*_\sigma Z\to G$ and
$[\cdot,\cdot]_H:Z*_\rho Z\to H$ such that $x=\pre G[x,y]\cdot y$ for
all $(x,y)\in Z*_\sigma Z$, and $x\cdot [x,y]_H=y$ for all
$(x,y)\in Z*_\rho Z$.

We refer to \cite[\S2]{kqwrieffel} for the required background on
Banach bundles, \fb s, and \fb\ equivalence.  In particular, our \fb s
will all be saturated (so that each fibre is an \ibm).  If $p:\B\to G$
is a \fb, $q:\E\to T$ is a Banach bundle, and $T$ is a left $G$-space,
then $\B$ acts on the left of $\E$ if there is a continuous pairing
$(b,e)\mapsto b\cdot e$ from
$\B*\E=\set{(b,e)\in \B\times \E:s(p(b))=\rho(q(e))}$ to $\E$ that covers
the pairing $(\gamma,x)\mapsto \gamma\cdot x$ from $G*T$ to $T$ in the
sense that $q(b\cdot e)=p(b)\cdot q(e)$, is associative in the sense
that $b\cdot (c\cdot e)=(bc)\cdot e$, and is submultiplicative for the
norms in the sense that $\|b\cdot e\|\le \|b\|\|e\|$.  Right actions are
defined similarly.  If $\B\to G$ and $\CC\to H$ are \fb s, then a
Banach bundle $\E\to T$ over a $(G,H)$-equivalence $T$ is a
$\B\sme \CC$ equivalence if there are continuous commuting actions of
$\B$ and $\CC$ on $\E$ ($\B$ on the left and $\CC$ on the right) such
that the bimodule actions are associative in the sense that
$(b\cdot e)\cdot c=b\cdot (e\cdot c)$, there is a continuous
fibre-wise sesquilinear pairing $\lip L<\cdot,\cdot>:\E*_\sigma\E\to\B$
that covers the pairing $\pre G[\cdot,\cdot]$ in the sense that
$p_\B(\lip L<e,f>)=\pre G[q(e),q(f)]$, commutes with the pairing
$\B*\E\to\B$ in the sense that $\lip L<b\cdot e,f>=b\lip L<e,f>$, and
similarly for the right action, and each $E_x$ is a
$B_{\rho(x)}\sme C_{\sigma(x)}$ \ibm.

The crux of the \emph{Ladder Technique} mentioned in the introduction
is a straightforward observation about maps between sets: if
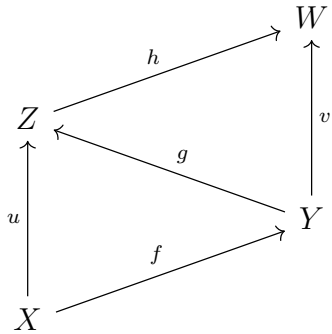
\[
  \begin{tikzcd}[column sep= 3cm]
    &W
    \\
    Z\arrow[ur,"h"]
    \\
    &Y \arrow[ul,"g"'] \arrow[uu,"v"']
    \\
    X \arrow[ur,"f"] \arrow[uu,"u"]
  \end{tikzcd}
\]
is a commutative diagram of sets and maps such that $u$ and $v$ are
bijections, then $f$, $g$, and $h$ are also bijections.  We refer to
such a diagram as a \emph{ladder}, with \emph{vertical struts} $u$ and
$v$ and \emph{rungs} $f$, $g$, and $h$. We refer to the preceding
observation as the \emph{Ladder Lemma}.

\section{Fell Bundle Ideals}
\label{sec:fell-bundle-ideals}

The notion of a Fell bundle ideal and the Rieffel correspondence for
equivalent Fell bundles is central here.  For details we refer to
\cite[Subsection~2.7 and Section~3]{kqwrieffel}.  We recall some of
the basics for convenience.  A Fell subbundle $\II$ of a \fb\
$\B\to G$ is an \emph{ideal} of $\B$ if $bc\in \II$ whenever
$(b,c)\in\B^{(2)}$ and either $b\in \II$ or $c\in \II$.  In
particular, an ideal of a Fell bundle is itself a Fell bundle and is
therefore closed under the adjoint operation.

Suppose that $q:\E\to T$ is a $\B\sme\CC$ Fell bundle equivalence.
The \emph{Rieffel correspondence} for \fb\ equivalences
\cite[Theorem~3.10]{kqwrieffel} says that there are lattice
isomorphisms between the ideals of $\B$, the full Banach $\B\sme\CC$
submodules of $\E$, and the ideals of $\CC$.  Specifically, if $\II$
is an ideal in $\B$, then the corresponding full Banach $\B\sme\CC$
submodule, $\M_{\JJ}$, is
\begin{equation}
  \label{eq:3}
  \II\cdot \E:= \bigcup_{s(h)=\rho(t)} J_{h}\cdot E_{t}.
\end{equation}
Describing the ideal $\II_{\M}$ in $\B$ corresponding to a full Banach
$\B\sme\CC$ submodule, $\M$, takes a bit more fussing.  Extending the
usual conventions for \ib s, we let $ \lip L<\M,\E>$ denote the span
of all possible left-inner products $\lip L<m,e>$ with $m\in \M$ and
$e\in \E$ as in Eq (3.2) in \cite{kqwrieffel}.  Then
\begin{equation}
  \label{eq:5}
  \II_{\M}= \coprod_{h\in H} \overline{\lip L<\M,\E>\cap B_{h}}
\end{equation}
where the closure is taken in the norm topology on $B_{h}$.

We can describe \eqref{eq:5} a bit more elegantly using a simple
observation (which probably should have been noted in
\cite{kqwrieffel}).

\begin{lemma}
  \label{lem-smallest} Suppose that $p:\B\to G$ is a Fell bundle and
  that $\mathscr S$ is any subset of $\B$.  Then there is a smallest
  ideal, $\Id_{\B}(\mathscr S)$, of $\B$ such that
  $\mathscr S\subset \Id_{\B}(\mathscr S)$.  We call
  $\Id_{\B}(\mathscr S)$ the ideal of $\B$ generated by $\mathscr S$.
\end{lemma}

\begin{proof}
  It will suffice to see that the arbitrary intersections of ideals is
  an ideal.  To this end, let $\II^{a}$ be an ideal in $\B$ for each
  $a\in A$.  By \cite[Proposition~2.30]{kqwrieffel},
  $I^{a}=\Gamma_{0}(\go;\II_{a})$ is a $G$-invariant ideal in
  $\Gamma_{0}(\go;\B)$, and
  $\II^{a}=\set{b\in\B:b^{*}b\in I^{a}_{s(b)}}$.  Then
  $I=\bigcap I^{a}$ is a $G$-invariant ideal in $\Gamma_{0}(\go;\B)$,
  and $\II=\set{b\in \B:b^{*}b\in I_{s(b)}}$ is an ideal in $\B$ by
  \cite[Proposition~2.30]{kqwrieffel}.  Now it is clear that
  $\II=\bigcap \II^{a}$.
\end{proof}

Now it is not hard to see that
\begin{equation}
  \label{eq:6}
  \II_{\M}=\Id_{\B}(\lip L<\B,\M>).
\end{equation}

\begin{remark}
  Of course, it is possible in the proof of Lemma~\ref{lem-smallest}
  that $I=\bigcap I^{a}$ is the zero ideal.  But then
  $\JJ=\set{0_{x}:x\in G}$.
\end{remark}

\section{The vertical struts}\label{sec:struts}
Let $H$ be a wide closed subgroupoid of $M$.\footnote{Here ``wide''
  simply means that $\ho=M\units$.}  Since our standing assumptions
are that both $H$ and $M$ have open range (and source) maps, it
follows from \cite[Corollary~2.50]{tool} that $M$ is a
$(M\rtimes M/H, H)$ equivalence.  Specifically, the moment map for the
left-action is given by $\rho(m)=m H$ where we have identified the
unit space of $M\rtimes M/H$ with the orbit space $M/H$. Then
$(m,nH)\cdot n=mn$.  The moment map for the right action is
$\sigma(m)=s(m)$ and then $m\cdot h=mh$.  Note that if
$\rho(m)=\rho(n)$, then $nH=mH$ and there is a unique $h\in H$ such
that $mh=n$ and $[m,n]_{H}=h $.  On the other hand, if
$\sigma(m)=\sigma(n)$ then $(mn^{-1},nH)\cdot n=m$ and
$\pre {M\rtimes M/H}[m,n] = (mn^{-1},nH)$.

\begin{thm}\label{abstract left}
  Let $p:\B\to M$ be a \fb, and let $H$ be a closed subgroupoid of $M$
  that is wide in the sense that $H\units=M\units$.  
  Since $M$ acts on
  $M/H$ we can form the action Fell bundle $\B\rtimes M/H$ as defined
  following \cite[Remark~4.12]{hkqw2}.
  Then $\B$ is a
  $\B\rtimes M/H\sme \B|_H$ equivalence, with operations
  \begin{enumerate}
  \item left action: $(b,mH)\cdot c=bc$;

  \item right action: $b\cdot c=bc$;

  \item left inner product: $\lip L<b,c>=(bc^*,p(c)H)$;

  \item right inner product: $\rip R<b,c>=b^*c$.
  \end{enumerate}
\end{thm}

\begin{proof}
  The left-action defined in (a) clearly makes sense on
  $(\B\rtimes M/H)*\B$ and covers the $M\rtimes M/H$ action on $M$.
  Also, on the one hand,
  \begin{equation}
    \label{eq:1}
    \bigl((b,mH)(c,nH)\bigr)\cdot d= (bc,nH)\cdot d= bcd.
  \end{equation}
  While on the other hand
  \begin{equation}
    \label{eq:2}
    (b,mH)\cdot \bigl( (c,nH)\cdot d\bigr) = (b,mH)\cdot cd =bcd.
  \end{equation}
  Therefore the pairing is associative.  Since
  \begin{equation}
    \label{eq:4a}
    \| (b,mH)\cdot c\|=\|bc\|\le \|b\|\|c\|=\|(b,mH)\|\|c\|,
  \end{equation}
  the pairing is submultiplicative.  Therefore we have an action of
  the semidirect product Fell bundle $\B\rtimes M/H$ on the Banach
  bundle $\B$ as required.

  For the associativity of the bimodule actions, we have
  \begin{align*}
    \bigl((b,mH)\cdot c\bigr)\cdot d
    &=(bc)\cdot d
      =bcd
      =(b,mH)\cdot (c\cdot d).
  \end{align*}

  We show that the left pairing $\B*_\sigma\B\to \B\rtimes M/H$ covers
  the pairing $M*_\sigma M\to M\rtimes M/H$:
  \begin{align*}
    p\left(\lip L<b,c>\right)
    &=p(bc^*,p(c)H)
    \\&=\bigl(p(b)p(c)\inv,p(c)H\bigr)
    \\&=\pre{M\rtimes M/H}[p(b),p(c)].
  \end{align*}

  For the adjoint property of the left pairing, we have
  \begin{align}
    \lip L<b,c>^*
    &=(bc^*,p(c)H)^*
    \\&=(cb^*,p(bc^*)p(c)H)
    \\&=(cb^*,p(b)p(c)\inv p(c)H)
    \\&=(cb^*,p(b)H)
    \\&=\lip L<c,b>.
  \end{align}

  Next,
  \begin{align*}
    \lip L<(b,mH)\cdot c,d>
    &=\lip L<bc,d>
    \\&=(bcd^*,p(d)H)
    \\&=(b,mH)\cdot (cd^*,p(d)H)
    \\&= (b,mH)\cdot \lip L<c,d>,
  \end{align*}
  where the operations are valid because
  \[
    mH=p(c)H=p(cd^*)p(d)H.
  \]

  The verifications of the remaining axioms, involving the right
  pairing, are similar and easier.

  Finally, for $m\in M$ we need to verify that $B_m$ is a
  $(\B\rtimes M/H)_{\rho(m)}\sme (\B|_H)_{\sigma(m)}$ \ibm.
  Unraveling the definitions, we must show that $B_m$ is a
  $(B_{r(m)}\times \{mH\})\sme B_{s(m)}$ \ibm.  Since the left-module
  action of $B_{r(m)}\times \{mH\}$ on $B_m$ is essentially just the
  left action of $B_{r(m)}$, the axiom follows immediately from the
  fact that $B_m$ is a $B_{r(m)}\sme B_{s(m)}$ \ibm.
\end{proof}

Recall that if $G$ acts on $H$ by isomorphisms, then we can form the
semidirect product $H\propto G$ as in \cite[Definition~5.1]{hkqw1}.
Then if $p:\A\to H$ is a Fell bundle and $G$ acts on $\A$ by
isomorphisms we can form the semidirect product Fell bundle $\A\propto
G$ as defined prior to \cite[Lemma~5.5]{hkqw1}.

\begin{cor}\label{left}
  Let $G$ act by isomorphisms on a \fb\ $p\:\A\to H$.  Then
  $\A\propto G$ is an $\A\propto G\rtimes G\sme \A$-equivalence, with
  operations
  \begin{enumerate}
  \item \label{La} left action:
    $(a,\gamma,\eta)\cdot (b,\eta)=\bigl(a(\gamma \cdot
    b),\gamma\eta\bigr)$;

  \item \label{Lb} right action:
    $(b,\eta)\cdot c=\bigl(b(\eta\cdot c),\eta\bigr)$;

  \item \label{Lc} left inner product:
    $\lip L<{(a,\gamma)},{(b,\eta)}> =\bigl(a(\gamma\eta\inv\cdot
    b^*),\gamma\eta\inv,\eta\bigr)$;

  \item \label{Ld} right inner product:
    $\rip R<{(a,\gamma)},(b,\gamma)> =\gamma\inv\cdot(a^*b)$.
  \end{enumerate}
\end{cor}

\begin{proof}
  We appeal to \thmref{abstract left}, with $M=H\propto G$ and
  $\B=\A\propto G$.  Then we can identify $H$ with the wide
  subgroupoid $H*\go=\set{(h,u)\in H\times\go:\rho(h)=u}$ of the
  semidirect product $H\propto G$.  Then $H$ acts on the right of
  $M=H\propto G$ and $[h,x]\mapsto x$ is a homeomorphism of
  $(H\propto G)/H$ onto $G$ where as usual we have written $[h,x]$ for
  the $H$-orbit of $(h,x)\in H\propto G$.  Furthermore, this map
  intertwines the natural $H\propto G$-action on the quotient
  $(H\propto G)/H$ with the composition of the natural map on $H\propto G$
  onto $G$ with left translation.  Therefore the horizontal maps
  induced by $[h,x]\mapsto x$ induce a diagram
  \begin{equation}
    \label{eq:4}
    \begin{tikzcd}
      (\A\propto G)\rtimes (H\propto G)/H \arrow[d] \arrow[r]
      & (\A\propto G)\rtimes G \arrow[d] \\
      (H\propto G)\rtimes (H\propto G)/H \arrow[r] & (H\propto
      G)\rtimes G
    \end{tikzcd}
  \end{equation}
  which gives a Fell bundle isomorphism of
  $\B\rtimes M/H = (\A\propto G)\rtimes (H\propto G)/H$ with
  $(\A\propto G)\rtimes G$.

  Since $(\A\propto G)|_{H}$ is isomorphic (as a Fell bundle) to $\A$,
  we can realize $\A\propto G$ as a
  $(\A\propto G)\rtimes G\sme \A$-equivalence.  Then the formulas for
  the actions and inner products follow from routine computations.
\end{proof}

For convenience and easy reference, we record some results from
\cite{hkqw2}.  Recall from \cite[Proposition~6.3]{hkqw2} that if $G$
acts freely and properly on $H$ by isomorphism, then $H$ becomes a
$H\propto G\sme G\backslash H$ equivalence where the left-action of
$H\propto G$ on $H$ is given by $(h,x)\cdot k=h(x\cdot k)$ and the
right action of $G\backslash H$ on $H$ is given by
$h\cdot (G\cdot k)=hk$ provided $s(h)=r(k)$.  Then we have the
following.

\begin{thm} [{\cite[Theorem~7.2]{hkqw2}}]\label{abstract right} Let
  $\B\to H$ be a \fb, and let $G$ act on $\B$ principally by
  isomorphisms.  Then $\B$ is a $\B\propto G\sme G\under \B$
  equivalence, with operations
  \begin{enumerate}
  \item left action: $(b,\gamma)\cdot c=b(\gamma\cdot c)$;

  \item right action: $b\cdot (G\cdot c)=bc$;

  \item left inner product: $\lip L<b,c>=\bigl(b(x\cdot c^*),x)$ where
    $x=\pre G[s(b),s(c)]$;

  \item right inner product: $\rip R<b,c>=G\cdot b^*c$.
  \end{enumerate}
\end{thm}



Concretely, we can suppose that $K$ and $G$ act on (the left) of a
space $T$ such that the actions commute and the moment map
$\rho:T\to \ko$ is a free and proper $G$-bundle.  Then $G$ acts freely
and properly by isomorphisms on the action groupoid $K\rtimes T$ via
$x\cdot (k,t)=(k,x\cdot t)$.  Then we can identify
$G\backslash (H\rtimes T)$ with $K$ and $K\rtimes T$ is a
$(K\rtimes T)\propto G\sme K$ equivalence.  Then we have the
following.

\begin{cor}[{\cite[Theorem~7.1]{hkqw2}}] \label{right with T} Let $K$,
  $G$, and $T$ be as above.  Let $p:\A\to K$ be a \fb.  Then
  $\A\rtimes T$ is an $\A\rtimes T\propto G\sme \A$ equivalence, with
  operations
  \begin{enumerate}
  \item left action: $(a,t,x)\cdot (b,t')=(ab,x\cdot t')$;

  \item right action: $(a,t)\cdot b=(ab,p(b)\inv\cdot t)$;

  \item left inner product:
    $\lip L<{(a,t)},{(b,t')}> =(ab^*,p(b)\cdot t,x)$ where $x$ is the
    unique element of $G$ such that $t=x\cdot t'$;

  \item right inner product: $\rip R<{(a,t)},(b,t')> =a^*b$.
  \end{enumerate}
\end{cor}

Note that we have commuting left actions of $G$ on itself---the first
just left translation and the second given by $x\cdot y=yx^{-1}$.
Furthermore, the moment map $r:G\to\go$ for left translation is a free
and proper $G$-bundle for the action by right translation.  Thus, if
$p:\A\to G$ is a Fell bundle, then the corresponding $G$-action on the
action bundle $\A\rtimes G$ is given by $x\cdot
(b,y)=(b,yx^{-1})$. Then we have the following result.

\begin{cor}[{\cite[Theorem~8.1]{hkqw2}}]\label{right}
  Let $p:\A\to G$ be a \fb.  Then $\A\rtimes G$ is an
  $\A\rtimes G\propto G\sme \A$-equivalence, with operations
  \begin{enumerate}
  \item left action: $(a,x,y)\cdot (b,z)=(ab,zy\inv)$;

  \item right action: $(a,z)\cdot b=(ab,p(b)^{-1}z )$;

  \item left inner product:
    $\lip L<{(a,z)},{(b,y)}>=\bigl(ab^*,p(b)z,z\inv y)$;

  \item right inner product: $\rip R<(a,\zeta),(b,\gamma)> =a^*b$.
  \end{enumerate}
\end{cor}

Now that we have equivalent \fb s, by \cite[Theorem~3.10]{kqwrieffel}
we have an associated \emph{Rieffel correspondence}, giving a lattice
isomorphism between the sets of Fell-bundle ideals.  In the case that
a groupoid $G$ acts consistently on the \fb s, the Rieffel
correspondence restricts to the invariant ideals.  More precisely, we
have the following.

\begin{thm}\label{left strut}
  In \corref{left}, if $\II$ is a $G$-invariant ideal of $\A$, then
  the ideal of $\A\propto G\rtimes G$ associated to $\II$ via the
  Rieffel correspondence is $\II\propto G\rtimes G$.  In fact, the
  Rieffel correspondence restricts to a bijection between the
  $G$-invariant ideals of $\A$ and the $G$-invariant ideals of
  $\A\propto G\rtimes G$ 
  for the usual $G$-action on an action
    bundle 
  as above.
\end{thm}

\begin{proof}
  First we show that
  \begin{equation}\label{eq:8}
    (\A\propto G)\cdot\II=\II\propto G.
  \end{equation}
  If $(a,x)\in \A\propto G$ and $i\in \II$, then
  \begin{align}
    (a,x)\cdot i
    &=\bigl(a(x\cdot i),x\bigr).
  \end{align}
  Putting $j=a(x\cdot i)$, we have $j\in \II$ because $\II$ is a
  $G$-invariant ideal, and so $(j,x)\in \II\propto G$.

  For the opposite containment, consider $(i,x)\in \II\propto G$.  Let
  $h=p(i)$.  Then $i\in J_{h}=A_{h}\cdot J_{s(h)}$ by
  \cite{kqwrieffel}*{Lemma~2.26}.  Then $i=a j$ with $a\in A_{h}$ and
  $j\in J_{s(h)}$ by \cite{kqwrieffel}*{Remark~2.2}.  Note that left
  multiplication by $x$ is a bijection of $A\restr{H_{s(x)}}$ onto
  $A\restr{H_{r(x)}}$.\footnote{Recall that if $u\in \go$, then let
    $H_{u}=\rho^{-1}(u)$ where $\rho:H\to\go$ is the moment map.  See
    \cite{kqwrieffel}*{Definition~4.9}.}  Since $\II$ is
  $G$-invariant, left-multiplication by $x$ also induces a bijection
  of $\II\restr{H_{s(x)}}$ onto $\II\restr{H_{r(x)}}$.  In particular,
  $j=x\cdot j'$ for some $j'\in J_{s(h)}$.  Then
  \begin{equation}
    \label{eq:7}
    (i,x)=(a j,x)= (a, x)\cdot
    j'\in  (\A\propto G)\cdot \II.
  \end{equation}
  Hence \eqref{eq:8} holds as claimed.

  Now let $\JJ$ be the ideal of $\A\propto G\rtimes G$ associated to
  $\II$ via the Rieffel correspondence.
  By the above, to show that
  \[
    \JJ\subset (\II\propto G)\rtimes G
  \]
  it suffices to consider an element of the form
  $\lip L<{(a,x)},(i,y)>$ with $(a,x)\in \A\propto G$ and
  $(i,y)\in \II\propto G$.\footnote{Where we have written
    $\lip L<\cdot,\cdot>$ in place of the more cumbersome
    $\lip \A\propto G\rtimes G<\cdot,\cdot>$.}  Since the inner
  product is
  \begin{align*}
    \bigl(a(x y\inv\cdot i^*),x y\inv,y\bigr),
  \end{align*}
  and since
  \[
    a(x y\inv\cdot i^*)\in \II
  \]
  because $\II$ is a $G$-invariant ideal and ideals are self-adjoint,
  we see that
  \[
    \lip L<{(a,x)},{(i,y)}> \in \II\propto G\rtimes G.
  \]

  For the opposite containment, consider
  $(i,z,y)\in \II\propto G\rtimes G$.  Let $h=p(i)$.  Then
  $i\in J_{h}=A_{h}\cdot J_{s(h)}$.  Let $i=a\cdot j$ with
  $a\in A_{h}$ and $j\in J_{s(h)}$ as above.  We can assume
  $j=z\cdot j_{1}$.  Note that $(a,zy)\in \A\propto G$ while
  $(j_{1}^{*},y)\in \II\propto G$.  Hence
  \begin{equation}
    \label{eq:9}
    (i,z,y)=(aj,z,y)=(a(zyy^{-1}\cdot j_{1}),zyy^{-1},y)=\lip
    L<{(a,zy)},(j_{1}^{*},y)> \in \JJ.
  \end{equation}
  Hence $\JJ=\II\propto G\rtimes G$ as required.

  Now suppose that $\JJ$ is a $G$-invariant ideal in
  $\A\propto G\rtimes G$.  Then the corresponding ideal, $\II$, in
  $\A$ is generated by pairings of the form
  $\rip R<{(a,x,y)\cdot (b , y)},(c,xy)>$ with $(a,x,y)\in \JJ$ and
  $(b,y),(c,xy)\in \A\propto G$.  That is, $\II$ is generated by
  elements of the form
  \begin{equation}
    \label{eq:10}
    (xy)^{-1}\cdot \bigl[\bigl(a(x\cdot b)\bigr)^{*}c\bigr].
  \end{equation}
  But
  \begin{align}
    \label{eq:11}
    z\cdot \Bigl( (xy)^{-1}\cdot \bigl[\bigl(a(x\cdot
    b)\bigr)^{*}c\bigr]\Bigr)
    &= (xyz^{-1})^{-1}\cdot \bigl[\bigl(a(x\cdot
      b)\bigr)^{*}c\bigr]  \\
    &= \rip R <{(a,x,yz^{-1})\cdot (b , yz^{-1})},(c,xyz^{-1})>
  \end{align}
  Since the $G$-action is isometric and bijective on the fibres,
  it now follows from \eqref{eq:5} that $\II$ is a $G$-invariant in
  $\A$.

  Since $\II\propto G\rtimes G$ is clearly $G$-invariant in
  $\A\propto G\rtimes G$, and since the Rieffel correspondence is a
  bijection, the result follows.
\end{proof}

Similarly, starting with action-product \fb s, we have:

\begin{thm}\label{right strut}
  In \corref{right}, if $\II$ is an ideal of $\A$, then the ideal,
  $\JJ$, of $\A\rtimes G\propto G$ associated to $\II$ via the Rieffel
  correspondence is $\II\rtimes G\propto G$.
\end{thm}

\begin{proof}
  First we show that in analogy with \eqref{eq:8} we have
  \[
    (\A\rtimes G)\cdot\II=\II\rtimes G.
  \]
  If $(a,z)\in \A\rtimes G$ and $i\in \II$, then
  \begin{align}
    (a,z)\cdot i
    &=(ai,p(i)^{-1}z ),
  \end{align}
  which is in $\II\rtimes G$ because $\II$ is an ideal.

  For the opposite containment, let $(i,z)\in \II\rtimes G$ with
  $p(i)=x$.  Note that, $s(x)=r(z)$.  As in the proof of
  Theorem~\ref{left strut}, we can factor $i=aj$ with $a\in A_{x}$ and
  $j\in J_{s(x)}$.  Then
  \[
    (i,z)=(aj,z)= (a,z)\cdot j\in (\A\rtimes G)\cdot \II.
  \]

  Now let $\JJ$ be the ideal of $\A\rtimes G\propto G$ associated to
  $\II$ via the Rieffel correspondence.  For the containment
  \[
    \JJ\subset \II\rtimes G\propto G,
  \]
  let $(a,z)\in \A\rtimes G$ and $(i,y)\in \II\rtimes G$.  Then
  \begin{align*}
    \lip L<{(a,z)},{(i,y)}>
    &=\bigl(ai^*,p(i)z,z^{-1}y\bigr)
      \in \II\rtimes G\propto G
  \end{align*}
  because $\II$ is an ideal and hence self-adjoint.

  For the opposite containment, given
  $(i,z,y)\in \II\rtimes G\propto G$, we can factor $i=aj$ with
  $a\in A_{x}$ and $j\in J_{s(x)}$ where $x=p(i)$.  Since $s(x)=r(z)$,
  we have
  \[
    (i,z,y)=\lip L<{(a,z)},{(j^{*},zy)}> \in \JJ
  \]
  because $(j^{*},zy)\in \II\rtimes G$.
\end{proof}

\section{Ladders}\label{sec:ladder}

If $p:\A\to H$ is Fell bundle on which $G$ acts by isomorphisms, then
we let $\II_G(\A)$ be the set of $G$-invariant ideals of the \fb\ $\A$.
Then we let $\II_{G}(\A\propto G\rtimes G)$ be the $G$-invariant ideals of
the action Fell bundle $\A\propto G\rtimes G$ as above.  Consider the
diagram
\begin{equation}\label{eq:ladder}
  \begin{tikzcd}
    &\II(\A\propto G\rtimes G\propto G)\\
    \II_G(\A\propto G\rtimes G)
    \arrow[ur,"\KK\mapsto \KK\propto G"]\\
    &\II(\A\propto G) \arrow[ul,swap,"\JJ\mapsto \JJ\rtimes G"]
    \arrow[uu,"\JJ\mapsto \JJ\rtimes
    G\propto G"']\\
    \II_G(\A)  \arrow[ur,"\II\mapsto \II\propto G"']
    \arrow[uu,"\II\mapsto \II\propto G\rtimes G"]
  \end{tikzcd}
\end{equation}
This diagram commutes and constitutes our basic Ladder Diagram.
Moreover the vertical arrows coincide with
the Rieffel correspondences, by Theorems~\ref{left strut} and
\ref{right strut}.  Therefore by the Ladder lemma all the arrows are
bijections.

We record part of this formally:

\begin{thm}\label{fb ladder}
  If a groupoid $G$ acts on a \fb\ $\A\to H$, then every ideal of
    $\A\propto G$ is of the form $\II\propto G$ for an $G$-invariant
    ideal of $\A$ and the map
    \[
      \II\mapsto \II\propto G
    \]
    gives a lattice isomorphism between the $G$-invariant ideals of
    $\A$ and all ideals in $\A\propto G$.
\end{thm}

Note that the above ladder \eqref{eq:ladder}
could be continued indefinitely upward.
And it could be started with the structure on the right-hand side, so
that the first rung would be an action product (rather than a
semidirect product).

To see how that would look, we consider 
a Fell bundle $p:\A\to G$ over a groupoid $G$.\footnote{To 
  continue the ladder \eqref{eq:ladder}, just replace $G$ with
  $G\propto G$.}  We can form the following 
ladder
\[
  \begin{tikzcd}
    &\II_{G}(\A\rtimes G\propto G\rtimes G)\\
    \II(\A\rtimes G\propto G)
    \arrow[ur,"\KK\mapsto \KK\rtimes G"]\\
    &\II_{G}(\A\rtimes G)
    \arrow[ul,swap,"\JJ\mapsto \JJ\propto G"]
    \arrow[uu,swap, "\JJ\mapsto \JJ\propto G\rtimes G"]\\
    \II(\A) \arrow[ur,swap,"\II\mapsto \II\rtimes G"]
    \arrow[uu,"\II\mapsto \II\rtimes G\propto G"]
  \end{tikzcd}
\]
The argument that the right strut is a bijection between $G$-invariant
ideals is a special case of Theorem~\ref{left strut}.
Hence, we get a
companion result to \thmref{fb ladder}.
\begin{thm}\label{bonus ladder}
  If $\A\to G$ is a \fb, then the map
  \[\II\mapsto\II\rtimes G\]
  gives a lattice isomorphism between the ideals of $\A$
  and the $G$-invariant ideals of $\A\rtimes G$.
\end{thm}






\end{document}